\documentclass{amsart}
\usepackage{amsmath}
\usepackage{amssymb}
\usepackage{amsthm}
\usepackage{enumerate}
\usepackage[dvipdfmx]{graphicx}
\theoremstyle{definition}
\newtheorem{definition}{Definition}[section]
\theoremstyle{plain}
\newtheorem{lemma}[definition]{Lemma}

\newtheorem{theorem}[definition]{Theorem}
\newtheorem{proposition}[definition]{Proposition}
\newtheorem{corollary}[definition]{Corollary}
\theoremstyle{remark}
\newtheorem{remark}[definition]{Remark}

\newtheorem{example}[definition]{Example}

\newcommand{\myLpt}{\operatorname{lpt}}
\newcommand{\myIso}{\operatorname{iso}}
\newcommand{\mycl}{\operatorname{cl}}
\newcommand{\myint}{\operatorname{int}}
\newcommand{\myrank}{\operatorname{rank}}

\begin{document}
	
	\title[D-minimal structure which defines every sequence]{There exists a d-minimal expansion of the $\mathbb R$-vector space over $\mathbb R$ which defines every sequence}
	\author[M. Fujita]{Masato Fujita}
	\address{Department of Liberal Arts,
		Japan Coast Guard Academy,
		5-1 Wakaba-cho, Kure, Hiroshima 737-8512, Japan}
	\email{fujita.masato.p34@kyoto-u.jp}
	
	\begin{abstract}
		There exists a d-minimal expansion of the $\mathbb R$-vector space over $\mathbb R$ which defines every sequence.
		In this paper, we prove this assertion and the following more general assertion:
		Let $\mathcal R$ be either the ordered $\mathbb R$-vector space structure over $\mathbb R$ or the ordered group of reals.
		A first-order expansion of $\mathcal R$ by a countable subset $D$ of $\mathbb R$ and a compact subset $E$ of $\mathbb R$ of finite Cantor-Bendixson rank is d-minimal if $(\mathcal R,D)$ is locally o-minimal.
	\end{abstract}
	
	
	
	\maketitle

\section{Introduction} \label{sec:intro}

The topologies made by sets definable in o-minimal structures are tame \cite{vdD}.
They have dimension functions satisfying simple formulas and any unary definable function is continuous except finitely many points.
They are a part of tame properties enjoyed by o-minimal structures.
Several relatives of o-minimal structures has been proposed;  for instance, weakly o-minimal structures \cite{MMS} and locally o-minimal structures \cite{TV}.
 They are excepted to have tame topologies which are slightly wilder than that of o-minimal structures.

D-minimal structures \cite{Miller-dmin} are ones of them.
Studies on tameness of d-minimal structures are found in \cite{Fornasiero,Miller-dmin,MT2,T,HM}.
Several examples of d-minimal structures are known such as \cite{vdD2,FM2,MT}.
Let us recall the definition of d-minimal structures.
\begin{definition}[\cite{Miller-dmin}]
	An expansion $\mathcal R$ of an ordered set of reals $(\mathbb R,<)$ is \textit{d-minimal} if, either one of (consequently, both of) the following equivalent conditions is satisfied:
	\begin{itemize}
		\item For every $\mathcal M \equiv \mathcal R$, every subset of $M$ definable in $\mathcal M$ is the union of an open set and finitely many discrete set, where $M$ is the universe of $\mathcal M$.
		\item For every nonnegative integer $n$ and definable set $A$ of $\mathbb R^{n+1}$, there exists a positive integer $N$ such that each $A_x:=\{y \in \mathbb R\;|\; (x,y) \in A\}$ either has a nonempty interior or is the union of at most $N$ discrete sets.  
	\end{itemize}
\end{definition}
In this paper, we only consider the case in which the universe is the set of reals $\mathbb R$.
A general definition of d-minimality is found in \cite{Fornasiero}.

The main theorem of this paper is as follows:
\begin{theorem}\label{thm:Caushy_version}
	The structure $(\mathbb R, <,+,0,\Lambda, E)$ is d-minimal, where $\Lambda$ is the set of all real scalar functions on $\mathbb R$, $E=\{a_n\;|\; n \in \mathbb N\}$ and $(a_n)_{n \in \mathbb N}$ is a Cauchy sequence.
\end{theorem}
We simply call a subset of $\mathbb R$ a \textit{Cauchy sequence} by abuse of terminology when this set is of the form $\{a_n\;|\; n \in \mathbb N\}$, where $(a_n)_{n \in \mathbb N}$ is a Cauchy sequence.

The above theorem is a corollary of a more general assertion proved in this paper.
In order to describe the assertion, we need some more preparation. 
Firstly, we need the definition of locally o-minimal structures.
\begin{definition}
	An expansion of dense linear order without endpoints $(M,<,\ldots)$ is \textit{locally o-minimal} if, for every definable subset $X$ of $M$ and for every point $a\in M$, there exists an open interval $I$ such that $a \in I$ and $X \cap I$ is a finite union of points and 
	open intervals \cite{TV}.
\end{definition}

We next recall the notation used in \cite{FM}.
\begin{definition}
	Let $\mathcal R$ be an o-minimal expansion of $(\mathbb R,<,+)$ and $E$ be a subset of $\mathbb R$.
	Let $(\mathcal R,E)^{\#}$ denote the structure $(\mathcal R, (S))$, where $S$ ranges over all nonempty subsets of all Cartesian products $E^k$ ($k$ ranges over all positive integers). 
	It is obvious that $(\mathcal R,E)$ is a reduct of $(\mathcal R,E)^{\#}$. 
\end{definition}

We are now ready to introduce the second main theorem of this paper.
\begin{theorem}\label{thm:semi-d-minimal}
	Let $\mathcal R$ be either $(\mathbb R,<,+,0)$ or $(\mathbb R, <,+,0,\Lambda)$, where $\Lambda$ is the set of all real scalar functions on $\mathbb R$.
	Let $D$ be a countable subset of $\mathbb R$ such that $(\mathcal R,D)$ is locally o-minimal.
	Let $E$ be a compact subset of $\mathbb R$ which is a union of finitely many discrete sets. 
	Then $(\mathcal R,D \cup E)^{\#}$ is d-minimal. 
\end{theorem}

Hieronymi and Walsberg proposed a classification scheme for expansions of the ordered group of reals in \cite{HW}.
Under this classification scheme, the structure $(\mathcal R,D \cup E)^{\#}$ is categorized in `type A, affine' by \cite[Theorem A]{HW} and the following theorem:

\begin{theorem}\label{thm:non_field_type}
	Let $\mathcal R$, $D$ and $E$ be as in Theorem \ref{thm:semi-d-minimal}.
	Let $U$ be an open subset of $\mathbb R^n$ and $f:U \to \mathbb R^m$ be a continuous function definable in $(\mathcal R, D \cup E)^{\#}$. 
	Then there exists a dense open subset $V$ of $U$ definable in $(\mathcal R, D \cup E)^{\#}$ such that the restriction of $f$ to $V$ is locally affine. 
\end{theorem}
Here, a map $f:U \to \mathbb R^m$ is called \textit{locally affine} if, for every point $x \in U$, there exists an open neighborhood $V$ of $x$ in $U$, the restriction of $f$ to $V$ coincides with the restriction of an affine map to $V$.

The paper is organized as follows:
In Section \ref{sec:FM_results}, we introduce a sufficient condition for an o-minimal expansion of $(\mathbb R,<,+,0)$ by a subset $E$ of $\mathbb R$ to be d-minimal.
It is deduced from the results in \cite{FM}.
We also recall several basic properties of d-minimal expansions of the ordered group of reals in this section.
Section \ref{sec:main} is a main part of this paper.
We prove Theorem \ref{thm:semi-d-minimal} and Theorem \ref{thm:non_field_type} in this section.
A simple question related to Theorem \ref{thm:semi-d-minimal} is what occurs if we drop the assumption that $E$ is compact. 
In Section \ref{sec:counterexamples}, we give examples which indicate that the expansion is not necessarily d-minimal if the compactness assumption is dropped.
We give concluding remarks in Section \ref{sec:concluding}.

In the last of this section, we introduce the terms and notations used in this paper.
Throughout, `definable' means `definable with parameters.'
For a subset $X$ of $\mathbb R^n$, the notations $\myint(X)$ and $\mycl(X)$ denote the interior and the closure of $X$ in $\mathbb R^n$, respectively.

\section{Results by Friedman and Miller}\label{sec:FM_results}
Friedman and Miller provides a sufficient condition for an o-minimal structure over $\mathbb R$ by a subset $E$ of $\mathbb R$ to satisfy the condition that a definable subset of $\mathbb R$ with empty interior is discrete in \cite{FM}.
We can easily get a sufficient condition for it to be d-minimal by using the facts proven in \cite{FM}.
The initial purpose of this section is to introduce the sufficient condition and prove it.

We start with trivial lemmas.

\begin{lemma}\label{lem:basic0}
	Let $X$ be a subset of a topological space $T$.
	If $X$ is a union of finitely many discrete subsets of $T$, then $X$ has an isolated point.
\end{lemma}
\begin{proof}
	This seems to be a well known fact, but we give a complete proof here.
	
	Assume that $X$ is the union of $n$ many discrete subsets $S_1, \ldots, S_n$ of $T$.
	We prove the lemma by induction on $n$.
	The lemma is trivial when $n=1$.
	When $n>1$, the union $X':=\bigcup_{i=1}^{n-1}S_i$ has an isolated point, say $x$, by the induction hypothesis.
	There exists an open neighborhood $N$ of $x$ in $T$ such that $N \cap X'=\{x\}$.
	If $x$ is also isolated in $X=\bigcup_{i=1}^n S_i$, we have nothing to prove.
	Assume that $x$ is not isolated in $X$. 
	We can pick a point $x' \in S_n \cap N$ with $x \neq x'$.
	We can take an open neighborhood $N'$ of $x'$ in $T$ such that $N' \cap S_n=\{x'\}$ because $S_n$ is discrete.
	The intersection $N \cap N'$ is an open neighborhood of $x'$ and we have $X \cap (N \cap N')=\{x'\}$.
	This means that $x'$ is an isolated point in $X$.    
\end{proof}

\begin{lemma}\label{lem:basic2}
	Let $X$ be a subset of a topological space $T$ and $N$ be a positive integer.
	If $X$ is the union of $N$ discrete subsets of $T$, any subset of $X$ is the union of at most $N$ discrete subsets of $T$.
\end{lemma}
\begin{proof}
	The lemma is obvious because a subset of a discrete set is again discrete.
\end{proof}

In \cite{FM}, Friedman and Miller introduce the collections of subsets of $\mathbb R^n$ called $\mathcal S_n$ and $\mathcal T_n$.
We only need the definition of $\mathcal S_n$ in this paper.

\begin{definition}\label{def:weak_cell}
	Let $\mathcal R$ be an o-minimal expansion of the ordered set of reals.
	A subset $E$ of $\mathbb R$ is called \textit{sparse} (with respect to $\mathcal R$) if, for each positive integer $n$ and each function $f:\mathbb R^n \to \mathbb R$ definable in $\mathcal R$, the image $f(E^n)$ has an empty interior.

	A subset $X$ is called a \textit{weak cell} if $X$ is definable in $\mathcal R$ and one of the following forms:
	\begin{enumerate}
		\item[(i)] $B \times \mathbb R$;
		\item[(ii)] $\{(x,t) \in \mathbb R^{n+1}\;|\; x \in B \text{ and } f(x)=t\}$;
		\item[(iii)] $\{(x,t) \in \mathbb R^{n+1}\;|\; x \in B \text{ and } f(x)<t\}$;
		\item[(iv)] $\{(x,t) \in \mathbb R^{n+1}\;|\; x \in B \text{ and } t<g(x)\}$;
		\item[(v)] $\{(x,t) \in \mathbb R^{n+1}\;|\; x \in B \text{ and } f(x)<t<g(x)\}$,
	\end{enumerate}
	where $B \subseteq \mathbb R^n$ and $f,g:\mathbb R^n \to \mathbb R$ are definable in $\mathcal R$.
	
	Let $E$ be a sparse subset of $\mathbb R$.
	Consider a subset $A$ of $\mathbb R^{n+1}$.
	We say that $A \in \mathcal S_{n+1}$ if there exist a natural number $m$, a weak cell $C \subseteq \mathbb R^{m+n+1}$ and indexed family $(P_{\alpha})_{\alpha \in I}$ of subsets of $E^m$ such that $$A=\bigcup_{\alpha \in I} \bigcap_{u \in P_{\alpha}} C_u.$$
	Here, we consider $\mathbb R^0$ is a singleton.
\end{definition}

We collect two assertions from \cite{FM}.

\begin{proposition}\label{prop:FM-main}
	Consider an o-minimal expansion of the ordered set of reals $\mathcal R$ and a sparse subset $E$ of $\mathbb R$.
	Assume that, for every $m$ and every function $f:\mathbb R^n \to \mathbb R$ definable in $\mathcal R$, the closure of the image $f(E^m)$ is nowhere dense.
	Every subset of $\mathbb R^{n+1}$ definable in $(\mathcal R, E)^{\#}$ is a finite union of elements in $\mathcal S_{n+1}$.
\end{proposition}
\begin{proof}
	The proposition follows from \cite[Theorem A]{FM} and its proof.
\end{proof}
Note that it is not proven that an element in $\mathcal S_{n+1}$  is definable in $(\mathcal R,E)^{\#}$ in \cite{FM}.

\begin{lemma}\label{lem:FM-claim}
	Let $\mathcal R$ and $E$ be as in Proposition \ref{prop:FM-main}.
	Let $A \in \mathcal S_{n+1}$.
	Suppose that, for every $x \in \mathbb R^n$, the fiber $A_x:=\{y \in \mathbb R\;|\; (x,y) \in A\}$ has empty interiors.
	Then there exists a natural number $m$ and a function $f:\mathbb R^{m+n} \to \mathbb R$ definable in $\mathcal R$, such that the inclusion $A_x \subseteq \mycl(f(E^m \times \{x\}))$ holds for every $x \in \mathbb R^n$.
\end{lemma}
\begin{proof}
	See the claim in \cite[p.62]{FM} .
\end{proof}

We now get the following theorem:
\begin{theorem}\label{thm:FM_results}
	Let $\mathcal R$ be an o-minimal expansion of $(\mathbb R,<,+)$.
	Let $E$ be a sparse subset of $\mathbb R$.
	Assume that, for each $m,n \in \mathbb N$ and $f:\mathbb R^{m+n} \to \mathbb R$ definable in $\mathcal R$, there exists $N \in \mathbb N$ such that the closure of $f(E^m \times \{x\})$ is the union of at most $N$ discrete sets for each $x \in \mathbb R^n$.
	Then, the structure $(\mathcal R,E)^{\#}$ is d-minimal.
\end{theorem}
\begin{proof}
	The theorem is almost proven in \cite{FM}.
	In the proof, the word `definable' means `definable in $(\mathcal R,E)^{\#}$'.
	
	Let $X$ be a definable subset of $\mathbb R^{n+1}$.
	We have only to show that there exists $N>0$ such that $X_x:=\{y \in \mathbb R\;|\; (x,y) \in X\}$ either has a nonempty interior or is the union of at most $N$  discrete sets for each $x \in \mathbb R^n$.
	We may assume that $X_x$ has an empty interior for each $x \in \mathbb R^n$ by replacing $X$ with $\{(x,y) \in X\;|\; X_x \text{ has an empty interior}\}$.
	By Proposition \ref{prop:FM-main}, every definable subsets of $\mathbb R^{n+1}$ is the union of finitely many elements of $\mathcal S_{n+1}$.
	We may assume that $X$ is an element of $\mathcal S_{n+1}$ without loss of generality.
	
	Since $X$ is an element of $\mathcal S_{n+1}$ and $X_x$ has an empty interior for each $x \in \mathbb R^n$, Lemma \ref{lem:FM-claim} implies that there exist a nonnegative integer $m$ and a map $f:\mathbb R^{m+n} \to \mathbb R$ definable in $\mathcal R$ such that $X_x$ is contained in the closure $\mycl(f(E^m \times \{x\}))$ of $f(E^m \times \{x\})$ in $\mathbb R$ for each $x \in \mathbb R^n$.
	
	By the assumption, there exists a natural number $N$ such that the closure $\mycl(f(E^m \times \{x\}))$ is the union of at most $N$ discrete sets.
	Since $X_x$ is contained in $\mycl(f(E^m \times \{x\}))$, the fiber $X_x$ is also the union of at most $N$ discrete sets by Lemma \ref{lem:basic2}.
\end{proof}

We consider the cases in which the o-minimal structure $\mathcal R$ in Theorem \ref{thm:FM_results} is either the $\mathbb R$-vector space structure over $\mathbb R$  or $(\mathbb R, <, +, 0)$.
Using Theorem \ref{thm:FM_results}, we can prove the following proposition:
\begin{proposition}\label{prop:dmin_suff}
	Let $\mathcal R$ be the o-minimal structure and $\mathbb K$ be  the field given in either (1) or (2) below.
	\begin{enumerate}
		\item[(1)] $\mathcal R=(\mathbb R, <, +, 0)$ and $\mathbb K=\mathbb Q$.
		\item[(2)] Let $\mathcal R$ be the $\mathbb R$-vector space structure over $\mathbb R$ and $\mathbb K=\mathbb R$.
	\end{enumerate}
	Let $E$ be a subset of $\mathbb R$.
	The expansion $(\mathcal R,E)^{\#}$ is d-minimal if, for any nonnegative integer $n$ and any linear map $v:\mathbb R^n \to \mathbb R$ with coefficients in $\mathbb K$, the image $v(E^n)$ is closed and a finite union of discrete subsets of $\mathbb R$.
\end{proposition}
\begin{proof}
	Let $f:\mathbb R^{m+n} \to \mathbb R$ be a function definable in $\mathcal R$.
	We have only to prove that there exists a positive integer $N$ such that $\mycl_{\mathbb R}(f(E^m \times \{x\}))$ is the union of at most $N$  discrete sets for each $x \in \mathbb R^n$ by Theorem \ref{thm:FM_results}.
	
	When $\mathcal R$ is the $\mathbb R$-vector space structure over $\mathbb R$,
	by \cite[Chapter 1, Corollary 7.6]{vdD}, there exists a partition $\mathbb R^{m+n} =\bigcup_{i=1}^k A_i$ into basic semilinear sets such that the restriction $f|_{A_i}$ of $f$ to $A_i$ is affine for each $1 \leq i\leq k$.
	We can prove the assertion similar to \cite[Chapter 1, Corollary 7.6]{vdD} even when $\mathcal R=(\mathbb R, <, +, 0)$.
	In this case, each definable affine map has rational coefficients and real constant.
	
	Let $g_i$ be the affine function such that $f|_{A_i}=g_i|_{A_i}$.
	There exist a linear map $v_i:\mathbb R^m \to \mathbb R$ with coefficients in $\mathbb K$ and an affine map $w_i:\mathbb R^n \to \mathbb R$ such that $g_i=v_i+w_i$ for each $1 \leq i \leq k$.
	By the assumption, $v_i(E^m)$ is closed and it is a union of $N_i$ discrete sets for some positive integer $N_i$.
	The image $g_i(E^m \times \{x\})=v_i(E^m)+w_i(x)$ is closed and it is the union of at  most $N_i$ discrete sets.
	Set $N=\sum_{i=1}^k N_i$.
	Since $f(E^m \times \{x\}) \subseteq \bigcup_{i=1}^k g_i(E^m \times \{x\})$ and $\bigcup_{i=1}^k g_i(E^m \times \{x\})$ is closed, we have $\mycl(f(E^m \times \{x\})) \subseteq \bigcup_{i=1}^k g_i(E^m \times \{x\})$.
	The latter set $\bigcup_{i=1}^k g_i(E^m \times \{x\})$ is the union of at most $N$  discrete sets.
	Therefore, the subset $\mycl(f(E^m \times \{x\}))$ of $\bigcup_{i=1}^k g_i(E^m \times \{x\})$ is the union of at most $N$ discrete sets by Lemma \ref{lem:basic2}.
	We have finished the proof.
\end{proof}

We recall several basic facts on d-minimal structures in the last of this section.
They are collected from Miller's papers \cite{Miller-dmin} and \cite{Miller-choice}.
\begin{proposition}\label{prop:Miller}
	Consider a d-minimal expansion of the ordered group of reals.
	\begin{enumerate}
		\item[(1)] Let $A$ and $B$ be definable subsets of $\mathbb R^n$ with empty interiors.
		The union $A \cup B$ has an empty interior.
		\item[(2)] For any definable function $f:U \to \mathbb R$ on a definable open subset $U$ of $\mathbb R^n$, there exists a definable dense open subset $V$ of $U$ such that the restriction of $V$ is continuous.
		\item[(3)] Let $X$ be a definable subset of $\mathbb R^{m+n}$ and $\pi:\mathbb R^{m+n} \to \mathbb R^n$ be a coordinate projection.
		There exists a definable map $\tau:\pi(X) \to X$ such that the composition $\pi \circ \tau$ is the identity map on $\pi(X)$.
	\end{enumerate}
\end{proposition}
\begin{proof}
	The assertions are found in \cite[Section 7, Main Lemma]{Miller-dmin}, \cite[Section 8.3]{Miller-dmin} and \cite{Miller-choice}.
	The original assertions in Miller's paper are proven under more relaxed conditions.
	See the original papers for the original statements. 
\end{proof}

\section{Expansion of a compact set of finite Cantor-Bendixson rank}\label{sec:main}

We prove Theorem \ref{thm:semi-d-minimal} and Theorem \ref{thm:non_field_type} in this section.
We first recall the Cantor-Bendixson rank of a set and its basic properties.

\begin{definition}[\cite{FM2}]\label{def:lpt}
	We denote the set of isolated points in $S$ by $\myIso(S)$ for any topological space $S$.
	We set $\myLpt(S):=S \setminus \myIso(S)$.
	In other word, a point $x \in S$ belongs to $\myLpt(S)$ if and only if $x \in \mycl_S(S \setminus \{x\})$, where $\mycl_S(\cdot)$ denotes the closure in $S$.
	
	Let $X$ be a nonempty closed subset of a topological space $S$.
	We set $X\langle 0 \rangle=X$ and, for any $m>0$, we set $X \langle m \rangle = \myLpt(X \langle m-1 \rangle)$.
	We say that $\myrank(X)=m$ if $X \langle m \rangle=\emptyset$ and $X\langle m-1 \rangle \neq \emptyset$.
	We say that $\myrank X = \infty$ when $X \langle m \rangle \neq \emptyset$ for every natural number $m$.
	We call $\myrank X$ the \textit{Cantor-Bendixson rank} of $X$.
	We set $\myrank(Y):=\myrank(\mycl(Y))$ when $Y$ is a nonempty subset of $S$ which is not necessarily closed.
\end{definition}

\begin{lemma}\label{lem:very_basic}
	For a countable closed subset $A$ of $\mathbb R$, $\myrank(A)=k$ if and only if $k$ is the least number of discrete sets whose union is $A$.
\end{lemma}
\begin{proof}
	See \cite[1.3]{FM2}.
\end{proof}

\begin{corollary}\label{cor:very_basic}
	Let $\mathcal R$ be an expansion of a dense linear order without endpoints.
	If a definable set is a union of finitely many discrete set, it is a union of finitely many definable discrete sets.
\end{corollary}
\begin{proof}
	The corollary follows from Lemma \ref{lem:very_basic} because $\myIso(X)$ is definable when $X$ is a definable set.
\end{proof}

We next prove two key lemmas.
\begin{lemma}\label{lem:basic1}
	Let $m$ be a positive integer.
	Let $E_i$ be a nonempty compact subset of $\mathbb R$ which is the union of at most $(N_i+1)$ discrete sets for $1 \leq i \leq m$.
	Here, $N_i$ are nonnegative integers for $1 \leq i \leq m$.
	Then there exists a positive integer $K$ satisfying the following condition:
	
	For any linear map $v:\mathbb R^n \to \mathbb R$ with real coefficients, the image $v(\prod_{i=1}^mE_i)$ is compact and the union of at most $K$ discrete sets.
\end{lemma}
\begin{proof}
	We fix a linear map $v:\mathbb R^n \to \mathbb R$.
	Set $X= v(\prod_{i=1}^mE_i)$ for simplicity of notation.
	The set $X$ is compact because it is the image of the compact set $\prod_{i=1}^mE_i$ under a continuous map.
	We prove that $X$ is the union of at most $K$ discrete sets, where $K$ is a positive integer determined only by $m$, $N_1,\ldots, N_m$.
	We first show the following claim:
	\medskip
	
	\textbf{Claim 1.} Let $(x_n)_{n \in \mathbb N}$ be a Cauchy sequence of elements in $X$.
	Let $r=\lim_{n \to \infty}x_n$.
	Then, at least one of the following holds:
	\begin{enumerate}
		\item[(i)]  There are infinitely many $n \in \mathbb N$ such that $x_n=r$.
		\item[(ii)] There exists an infinite subsequence $(y_n)_{n \in \mathbb N}$ of $(x_n)_{n \in \mathbb N}$ such that $y_n \neq r$ for each $n \in \mathbb N$ and the limit $r$ is of the form $r=v(f_1,\ldots, f_m)$ with $f_i \in E_i$ and $f_k \in \myLpt(E_k)$ for some $1 \leq k \leq m$.
	\end{enumerate}
	\begin{proof}[Proof of Claim 1]
		We assume that the sequence $(x_n)_{n \in \mathbb N}$ does not possess property (i).
		We demonstrate that it has property (ii) under this assumption. 
		Since $(x_n)_{n \in \mathbb N}$ does not possess property (i), we may assume that $x_n \neq r$ for each $n \in \mathbb N$ by taking a subsequence if necessary.
		For each $n \in \mathbb N$ and $1 \leq i \leq m$, there exist $e_{ni} \in E_i$ such that $x_n=\sum_{i=1}^mb_ie_{ni}$ because $x_n$ is an element in $X$.
		For each $1 \leq i \leq m$, we set $D_i:=\{e_{ni}\;|\; n \in\mathbb N\}$.
		When $D_1$ is a finite set, there exists $r_1 \in E_1$ such that $\{n \in \mathbb N\;|\; e_{1n}=r_1\}$ is an infinite set.
		We can take a subsequence of $(x_n)$ so that $e_{in}=r_1$ for each $n \in \mathbb N$.
		When $D_1$ is an infinite set, we can take a subsequence $(n_j)_{j \in \mathbb N}$ so that $(e_{1n_j})_{j \in \mathbb N}$ is a Cauchy sequence because $E_1$ is a compact set.
		Let $f_1$ be the limit of the Cauchy sequence $(e_{1n_j})_{j \in \mathbb N}$.
		If there are infinitely many $j \in \mathbb N$ such that $e_{1n_j}=f_1$, by taking a subsequence, we may assume that $e_{1n_j}=f_1$ for all $j \in \mathbb N$.
		Consider the other case.
		By removing finitely many numbers in the sequence $\{n_j\}_{j \in \mathbb N}$, we may assume that $e_{1n_j} \neq f_1$ for all $j \in \mathbb N$.
		In this case, $f_1$ is an accumulation point of $E_1$.
		In summary, we have taken a subsequence $(f_{ni})_{n \in \mathbb N}$ of $(e_{ni})_{n \in \mathbb N}$ which satisfies only one of the following conditions for $i=1$:
		\begin{enumerate}
			\item[$(1)_i$] $f_{ni}=f_i$ for all $n \in \mathbb N$, where $f_i \in E_i$;
			\item[$(2)_i$] $\{f_{ni}\}_{n \in \mathbb N}$ is a Cauchy sequence converging to a point $f_i \in \myLpt(E_i)$ such that $f_{ni} \neq f_i$ for each $n \in \mathbb N$.
		\end{enumerate}
		Applying the same argument for $i>1$, we get a subsequence $(f_{ni})_{n \in \mathbb N}$ of $(e_{ni})_{n \in \mathbb N}$ enjoying only one of properties $(1)_i$ and $(2)_i$ for each $1 \leq i \leq m$.
		Set $y_n=v(f_{n1},\ldots, f_{nm})$ for each positive integer $n$.
		If the subsequence $(f_{ni})_{n \in \mathbb N}$ holds property $(1)_i$ for every $1 \leq i \leq m$, $y_n$ is a constant value independent of $n$.
		It means that $(x_n)_{n \in \mathbb N}$ enjoys property (i).
		It is a contradiction to the assumption.
		There exists $1 \leq k \leq m$ such that $(f_{nk})_{n \in \mathbb N}$ holds property $(2)_k$.
		We have $\lim_{n \to \infty} y_n=r$ because $(y_n)_{n \in \mathbb N}$ is a subsequence of $(x_n)_{n \in \mathbb N}$.
		We have $\lim_{n \to \infty} y_n=v(f_1\ldots, f_m)$ because $v$ is continuous.
		We have shown that, for all $1 \leq i \leq m$, there exist $f_i \in E_i$ such that $r=v(f_1,\ldots, f_m)$ and $f_k \in \myLpt(E_k)$ for some $1 \leq k \leq m$.
		%
	\end{proof}
	
	We next prove the following claim:
	\medskip
	
	\textbf{Claim 2.}
	Let $x=v(e_1,\ldots, e_m) \in X$, where $e_i \in E_i$ for each $1 \leq i \leq m$.
	The point $x$ is an accumulation point in $X$ if and only if $e_k \in \myLpt(E_k)$ for some $1 \leq k \leq m$.
	
	\begin{proof}[Proof of Claim 2]
		The `only if' part immediately follows from Claim 1.
		We prove the opposite implication.
		Assume that $e_k \in \myLpt(E_k)$ for some $1 \leq k \leq m$.
		Let $\{a_n\}_{n \in \mathbb N}$ be a sequence of elements in $E_k$ converging to $e_k$ satisfying the relation $a_n \neq e_k$ for each positive integer $n$.
		It is obvious that $\{x_n:=\sum_{1 \leq i \leq m, i \neq k}b_ie_i + b_ka_n\}$ is a Cauchy sequence of elements in $X \setminus \{x\}$ converging to $x$ when $v(x_1,\ldots, x_m)$ is given by $\sum_{i=1}^m b_ix_i$.
		It implies that $x$ is an accumulation point.
	\end{proof}
	
	We have finished the preparation.
	We prove the lemma by induction on $N:=\sum_{i=1}^mN_i$.
	When $N=0$, we have $N_i=0$ for all $1 \leq i \leq m$.
	It means that $E_i$ are compact and discrete for $1 \leq i \leq m$.
	Therefore, $E_i$ are finite sets for $1 \leq i \leq m$.
	In this case, the set $X$ is a finite set.
	In particular, it is closed and discrete.
	
	We next consider the case in which $N>0$.
	Set $X'_i=\{v(e_1,\ldots, e_m) \in X\;|\; e_j \in E \ (1 \leq j \leq m) \text{ and } e_i \in \myLpt(E_i)\}$ for each $1 \leq i \leq m$.
	Note that $\myLpt(E_i)$ is closed and it is the union of at most $N_i$ discrete sets by  Lemma \ref{lem:very_basic}.
	The set $X'_i$ is the union of at most $K'_i$ discrete sets by the induction hypothesis, where $K'_i$ is a natural number determined only by $m$, $N_1,\ldots N_m$.
	Set $Y=\{v(e_1,\ldots, e_m) \in X\;|\; e_j \in E \setminus \myLpt(E_j) \ (1 \leq j \leq m) \}$.
	It is discrete by Claim 2.
	Set $K=1+\sum_{i=1}^m K'_i$.
	Since $X=\bigcup_{i=1}^m X'_i \cup Y$, the set $X$ is the union of at most $K$  discrete sets.
	We have proven the lemma.
\end{proof}

We extend Lemma \ref{lem:basic1} as follows:
\begin{lemma}\label{lem:basic11}
	Let $\mathcal R$ be the o-minimal structure and $\mathbb K$ be  the field given in either (1) or (2) below.
	\begin{enumerate}
		\item[(1)] $\mathcal R=(\mathbb R, <, +, 0)$ and $\mathbb K=\mathbb Q$.
		\item[(2)] Let $\mathcal R$ be the $\mathbb R$-vector space structure over $\mathbb R$ and $\mathbb K=\mathbb R$.
	\end{enumerate}
	Let $D$ be a countable subset of $\mathbb R$ such that $(\mathcal R,D)$ is locally o-minimal.
	Let $m$ be a positive integer.
	Let $E$ be a nonempty compact subset of $\mathbb R$ which is a union of finitely many discrete sets.
	Let $v:\mathbb R^m \to \mathbb R$ be a linear map with coefficients in $\mathbb K$.
	Then the image $v((D\cup E)^m)$ is closed and a union of finitely many discrete sets.
\end{lemma}
\begin{proof}
	Set $X=v((D\cup E)^m)$.
	Let $v(x_1,\ldots, x_m)=\sum_{i=1}^m b_ix_i$.
	We may also assume that all $b_i$ are nonzero.
	For any $\sigma=(i_1,\ldots,i_m) \in \{0,1\}^m$, we set $$X_{\sigma}:=\left\{\sum_{j=1}^m b_je_j\;\middle|\; e_j \in E \text{ if }i_j=0 \text{ and } e_j \in D \text{ otherwise for each }1 \leq j \leq m\right\}.$$
	We obviously have $X=\bigcup_{\sigma \in \{0,1\}^m} X_{\sigma}$.
	Therefore, we have only to prove that $X_{\sigma}$ is closed and the union of finitely many discrete sets for each $\sigma \in \{0,1\}^m$.
	We fix an arbitrary $\sigma=(i_1,\ldots,i_m) \in \{0,1\}^m$.
	We may assume that $i_1=\cdots = i_k=0$ and $i_{k+1}=\cdots=i_m=1$ by permuting the coordinates if necessary.
	When $k=m$, the claim on $X_{\sigma}$ follows from Lemma \ref{lem:basic1}.
	We assume that $k<m$ for the rest of the proof.
	
	We set $Y_{\sigma}:=\{\sum_{j=1}^k b_je_j\;|\; e_j \in E\}$ and $Z_{\sigma}:=\{\sum_{j=k+1}^m b_je_j\;|\; e_j \in D\}$.
	We obviously have $X_{\sigma}=\{y+z\;|\; y \in Y_{\sigma}\text{ and }z \in Z_{\sigma}\}$.  
	The set $Y_{\sigma}$ is compact and it is a finite union of discrete sets by Lemma \ref{lem:basic1}.
	When $Z_{\sigma}$ is a finite set, $X_{\sigma}$ is closed and a union of finitely many discrete sets because $X_{\sigma}=\bigcup_{z \in Z_{\sigma}} z + Y_{\sigma}$ and $z+Y_{\sigma}:=\{z+y\;|\; y \in Y_{\sigma}\}$ is closed and it is a union of finitely many discrete sets for each $z \in Z_{\sigma}$.
	We consider the case in which $Z_{\sigma}$ is infinite in the rest of the proof. 
	Since $Z_{\sigma}$ is countable and definable in $(\mathcal R,D)$, the set  $Z_{\sigma}$ is discrete and closed by \cite[Lemma 2.3]{Fuji}.
	We have the following three separate cases:
	\begin{enumerate}
		\item[(a)] $\inf Z_{\sigma} > - \infty$ and $\sup Z_{\sigma} = \infty$;
		\item[(b)] $\inf Z_{\sigma} = - \infty$ and $\sup Z_{\sigma} < \infty$;
		\item[(c)] $\inf Z_{\sigma} = - \infty$ and $\sup Z_{\sigma} = \infty$.
	\end{enumerate}
	We only consider case (c).
	We can prove the lemma similarly in the other cases.
	
	We define a bijection $\iota:\mathbb Z \to Z_{\sigma}$ so that $\iota(n_1)<\iota(n_2)$ if and only if $n_1<n_2$.
	Fix an arbitrary point $x_0$ and put $\iota(0)=x_0$.
	When $n>0$ and $\iota(n)$ is determined, we define $\iota(n+1)=\inf\{x \in Z_{\sigma}\;|\;x>\iota(n)\}$.
	We have $\iota(n+1) \in Z_{\sigma}$ because $Z_{\sigma}$ is discrete and closed.
	We define $\iota(n)$ similarly when $n<0$.
	The injectivity of $\iota$ is obvious.
	Let $x$ be an arbitrary point in $Z_{\sigma}$.
	Since $Z_{\sigma}$ is closed and discrete, there are finitely many points in $Z_{\sigma}$ between $x_0$ and $x$.
	This fact implies the surjectivity of $\iota$.
	In cases (a) and (b), we construct order-preserving bijections $\iota:\mathbb N \to Z_{\sigma}$ and $\iota:(-\mathbb N) \to Z_{\sigma}$, respectively.
	
	Set $d=\inf\{|x-y|\;|\; x,y \in Z_{\sigma}, x \neq y\}$.
	The definable set $\{|x-y|\;|\; x,y \in Z_{\sigma}, x \neq y\}$ is discrete and closed by \cite[Proposition 2.8(1),(6)]{FKK}.
	It implies that $d>0$.
	Since $Y_{\sigma}$ is compact, by shifting $X_{\sigma}$ if necessary, there exists a positive integer $N$ such that $Y_{\sigma}$ is contained in the closed interval $[0,N]$.
	Since the ordered group of reals is archimedean, there exists a positive integer $M$ such that $Md>N$.
	We set $W_i=\{y + \iota(Mk+i)\;|\; y \in Y_{\sigma},\ k \in \mathbb Z\}$ for $1 \leq i \leq M$.
	We have $X_{\sigma}=\bigcup_{i=1}^M W_i$.
	Therefore, we have only to prove that $W_i$ is closed and a union of finitely many discrete sets.
	We fix $1 \leq i \leq M$.
	Let $k_1<k_2$ be arbitrary integers.
	By the definition of $d$, we have $\iota(Mk_2+i)-\iota(Mk_1+i) \geq M(k_2-k_1)d \geq Md$.
	Since $Y_{\sigma} \subseteq [0,N]$ and $Md>N$, we have
	\begin{equation}
		\iota(Mk_1+i)+Y_{\sigma} < \iota(Mk_2+i)+Y_{\sigma}. \label{eq:xxx}
	\end{equation}
	This inequality means that an arbitrary element in $\iota(Mk_1+i)+Y_{\sigma}$ is smaller than an arbitrary element in $\iota(Mk_2+i)+Y_{\sigma}$.
	The following equality is obvious:
	\begin{equation}
		W_i=\bigcup_{k \in \mathbb Z}(\iota(Mk+i)+Y_{\sigma}) \label{eq:xxx2}
	\end{equation}
	By Lemma \ref{lem:basic1}, $Y_{\sigma}$ is a union of finitely many discrete sets, say, $C_1,\ldots, C_K$.
	Relations (\ref{eq:xxx})  and (\ref{eq:xxx2}) imply that $W_i$ is closed because $Y_{\sigma}$ is closed.
	Since $C_j$ is contained in $Y_{\sigma}$ and $C_j$ is discrete, inequality (\ref{eq:xxx}) imply that the set $C'_j:=\{y + \iota(Mk+i)\;|\; y \in C_j,\ k \in \mathbb Z\}$ is discrete for each $1 \leq j \leq K$.
	The set  $W_i$ is the union of discrete sets $C'_1, \ldots C'_K$.
	We have finished the proof.
\end{proof}

We are now ready to complete the proof of Theorem \ref{thm:semi-d-minimal}.
\begin{proof}[Proof of Theorem \ref{thm:semi-d-minimal}]
	By Proposition \ref{prop:dmin_suff}, we have only to prove that, for any nonnegative integer $n$ and any linear map $v:\mathbb R^n \to \mathbb R$ with coefficients in $\mathbb K$, the image $v((D \cup E)^n)$ is closed and a finite union of discrete subsets of $\mathbb R$.
	It immediately follows from Lemma \ref{lem:basic11}.
\end{proof}

We give corollaries of Theorem \ref{thm:semi-d-minimal}.
The first one is Theorem \ref{thm:Caushy_version}.
\begin{proof}[Proof of Theorem \ref{thm:Caushy_version}]
	Set $c=\lim_{n \to \infty}a_n$.
	The set $E \cup \{c\}$ is compact, and it is the union of two discrete sets $E \setminus \{c\}$ and $\{c\}$.
	Apply Theorem \ref{thm:semi-d-minimal} to $\mathcal V=(\mathbb R, <,+,0,\Lambda)$ and $E \cup \{c\}$.
	The structure $(\mathcal V,E \cup \{c\})^{\#}$ is d-minimal.
	It is obvious that $(\mathcal V,E)$ is interdefinable with $(\mathcal V,E \cup \{c\})$ which is a reduct of $(\mathcal V,E \cup \{c\})^{\#}$.
\end{proof}

\begin{corollary}\label{cor:semi-d-minimal}
	Let $E$ be a compact subset of $\mathbb R$ which is the union of finitely many discrete sets. 
	Then $(\mathbb R,<,+,0,E,\mathbb Z)$ is d-minimal. 
\end{corollary}
\begin{proof}
	It follows from Theorem \ref{thm:semi-d-minimal} because $(\mathbb R,<,+,0,E)$  is a reduct of $(\mathbb R,<,+,0,E \cup \mathbb Z)^{\#}$ and $(\mathbb R, <, +,0, \mathbb Z)$ is locally o-minimal.
\end{proof}

The following remark points out that the o-minimal structure $\mathcal R$ in Theorem \ref{thm:semi-d-minimal} should be $(\mathbb R,<,+,0)$ when $D=\mathbb Z$.
\begin{remark}
	The structure $(\mathbb R, <,+,0,\Lambda, \mathbb Z)$ defines multiplication on $\mathbb R$ by \cite[Theorem B]{HT}.  
	Therefore, $(\mathbb R,<,+,0,\Lambda, E,\mathbb Z)$ is not d-minimal.
\end{remark}

We next begin to prove Theorem \ref{thm:non_field_type}.
We first prove two lemmas.
\begin{lemma}\label{lem:func_char}
	Let $\mathcal R$ be either $(\mathbb R,<,+,0)$ or $(\mathbb R, <,+,0,\Lambda)$, where $\Lambda$ is the set of all real scalar functions on $\mathbb R$.
	Let $D$ be a countable subset of $\mathbb R$ such that $(\mathcal R,D)$ is locally o-minimal.
	Let $E$ be a compact subset of $\mathbb R$ which is a union of finitely many discrete sets. 
	Let $X$ be a subset of $\mathbb R^n$ and $f:X \to \mathbb R$ be a function definable in $(\mathcal R, D \cup E)^{\#}$. 
	
	Then there are finitely many pairs $(g_i:\mathbb R^n \to \mathbb R,D_i)_{i=1}^k$ of linear maps $g_i$ and subsets $D_i$ of $\mathbb R$ such that, 
	\begin{enumerate}
		\item[(a)] for each $x \in X$, the equality $f(x)=g_i(x)+d$ holds for some $1 \leq i \leq k$ and $d \in D_i$ and
		\item[(b)] the sets $D_i$ are definable in $(\mathcal R, D \cup E)^{\#}$ and have empty interiors for all $1 \leq i \leq k$.
	\end{enumerate}
\end{lemma}
\begin{proof}
	We set $F=D \cup E$.
	Let $A$ be the graph of the definable function $f$.
	By Proposition \ref{prop:FM-main}, $A$ is the union of finitely many elements in $\mathcal S_{n+1}$, say $S_1,\ldots, S_l$.
	Let $\pi:\mathbb R^{n+1} \to \mathbb R^n$ be the projection forgetting the last coordinate.
	It is obvious that $\{\pi(S_i)\}_{1 \leq i \leq l}$ is a covering of $X$ and $S_i$ is the graph of the restriction $f_i$ of $f$ to $\pi(S_i)$.
	If we can take finitely many pairs of linear maps and subsets of $\mathbb R$ satisfying conditions (a) and (b) for $f_i$ for every $1 \leq i \leq k$, the union of the families of such pairs satisfies conditions (a) and (b) for $f$.
	Therefore, we may assume that $A$ is an element of $\mathcal S_{n+1}$. 
	Remember that $X$ and $f$ may not definable in $(\mathcal R, F)^{\#}$ under this assumption. 
	
	Since $A \in \mathcal S_{n+1}$, there exist a natural number $m$, a weak cell $C \subseteq \mathbb R^{m+n+1}$ and indexed family $(P_{\alpha})_{\alpha \in I}$ of subsets of $E^m$ such that $$A=\bigcup_{\alpha \in I} \bigcap_{u \in P_{\alpha}} C_u.$$
	The weak cell $C$ is of one of the forms (i) through (v) in Definition \ref{def:weak_cell}.
	Since $A$ is the graph of a function, the fiber $A_x=\{y \in \mathbb R\;|\; (x,y) \in A\}$ is a singleton for each $x \in \pi(A)$.
	We have $$A_x = \bigcap_{u \in P_{\alpha}}C_{(u,x)}$$ for some $\alpha \in I$.
	We can easily prove that $A_x$ is an infinite set when $C$ is either of the form (i), (iii) or (iv) in Definition \ref{def:weak_cell}.
	We only consider the case (iii) here.
	There exists a function $p:\mathbb R^m \times \mathbb R^n \to \mathbb R$ definable in $\mathcal R$ such that $C=\{(u,x,t) \in \mathbb R^{m+n+1}\;|\; p(u,x)<t\}$.
	Set $s:=\sup\{p(u,x)\;|\; u \in P_{\alpha}\} \in \mathbb R \cup \{+\infty\}$. 
	Every element in $A_x$ is not smaller than $s$.
	The supremum $s$ is finite because $A_x$ is not empty. 
	The fiber $A_x$ is an infinite set because any real number lager than $s$ belongs to $A_x$.
	It is absurd.
	Therefore, $C$ is either of the form (ii) or (v) in Definition \ref{def:weak_cell}.
	
	We first treat the easier case.
	Assume that $C$ is of the form (ii).
	There exists a function $p:\mathbb R^{m+n} \to \mathbb R$ definable in $\mathcal R$ such that $C=\{(u,x,t) \in \mathbb R^{m+n+1}\;|\; t = p(u,x)\}$.
	By \cite[Chapter 1, Corollary 7.6]{vdD}, there are finitely many affine functions $p_1,\ldots, p_k$ and a decomposition $\mathbb R^{m+n} = C_1 \cup \cdots \cup C_k$ into basic semilinear sets such that the restriction of $p$ to $C_i$ coincides with $p_i$ for every $1 \leq i \leq k$.
	Since $p_i$ are affine, we can take linear maps $g_i: \mathbb R^n \to \mathbb R$ and affine maps $h_i:\mathbb R^m \to \mathbb R$ such that $p_i(u,x)=g_i(x)+h_i(u)$ for every $u \in \mathbb R^m$ and $x \in \mathbb R^n$.
	Set $D_i=h_i(F^n)$.
	Then the pairs $(g_i,D_i)$ satisfy conditions (a) and (b).
	Take an arbitrary point $x \in \pi(A)$.
	We have $$A_x=\bigcap_{u \in P_{\alpha}}C_{(u,x)} = \bigcap_{u \in P_{\alpha}}\{p(u,x)\}$$ for some $\alpha \in I$.
	Take a point $u \in P_{\alpha}$. We have $A_x=\{p(u,x)\}$ because $A_x$ is not empty.
	The point $(u,x)$ belongs to $C_i$ for some $1 \leq i \leq k$.
	We obtain $p(u,x)=g_i(x)+h_i(u)$.
	We have shown that condition (a) is satisfied.
	Condition (b) is satisfied by Lemma \ref{lem:basic11}.
	
	We next consider the case in which $C$ is of the form $\{(u,x,t) \in \mathbb R^{m+n+1}\;|\; (u,x) \in B \text{ and }p(u,x)<t<q(u,x)\}$, where $B \subseteq R^{m+n}$ and $p,q:\mathbb R^{m+n} \to \mathbb R$ are definable in $\mathcal R$.
	Apply \cite[Chapter 1, Corollary 7.6]{vdD} to $p$.
	There are finitely many affine functions $p_1,\ldots, p_k$ and a decomposition $\mathbb R^{m+n} = C_1 \cup \cdots \cup C_k$ into basic semilinear sets such that the restriction of $p$ to $C_i$ coincide with $p_i$, for every $1 \leq i \leq k$.
	We can take linear maps $g_i: \mathbb R^n \to \mathbb R$ and affine maps $h_i:\mathbb R^m \to \mathbb R$ such that $p_i(u,x)=g_i(x)+h_i(u)$ for every $u \in \mathbb R^m$ and $x \in \mathbb R^n$.
	We set $D_i=h_i(F^m)$ for $1 \leq i \leq k$.
	Note that $D_i$ is closed by Lemma \ref{lem:basic11}.
	We want to show that the pairs $(g_i,D_i)$ satisfy conditions (a) and (b).
	
	The first target is to show that condition (a) is fulfilled.
	Fix an arbitrary point $x \in \pi(A)$.
	We have $$A_x=\bigcap_{u \in P_{\alpha}}C_{(u,x)}=\bigcap_{u \in P_{\alpha}}\{t \in \mathbb R\;|\; p(u,x) < t < q(u,x)\}$$ for some $\alpha \in I$.
	We set $P_{\alpha,i}=P_{\alpha} \cap \{u \in \mathbb R^m\;|\; (u,x) \in C_i\}$ for $1 \leq i \leq k$.
	Note that $p(u,x)=g_i(x)+h_i(u)$ when $u \in P_{\alpha,i}$.
	Recall that $A_x$ is a singleton.
	Let $t_x$ be the unique element in $A_x$.
	Set $s_i=\sup\{h_i(u)\;|\; u \in P_{\alpha,i}\}$ for $1 \leq i \leq k$.
	We have $t_x \geq g_i(x)+ s_i$ for $1 \leq i \leq k$.
	If the inequality $t_x>g_i(x) + s_i$ holds for every $1 \leq i \leq k$, every element between $\max\{g_i(x)+s_i\;|\; 1 \leq i \leq k\}$ and $t_x$ should belong to $A_x$.
	This contradicts the fact that $A_x$ is a singleton.
	We can find $1 \leq i \leq k$ such that $f(x)=g_i(x)+s_i$.
	It is obvious that $s_i \in \mycl(h_i(P_{\alpha,i})) \subseteq \mycl(h_i(F^m))=D_i$.
	
	The satisfaction of condition (b) follows from Lemma \ref{lem:basic11}.
\end{proof}

\begin{lemma}\label{lem:locally_constant}
	Consider a d-minimal expansion of the set of reals.
	Let $E$ be a definable subset of $\mathbb R$ having an empty interior and $U$ be a definable open subset of $\mathbb R^n$.
	Let $f:U \to E$ be a definable map.
	Then, there exists a definable open subset $V$ of $U$ such that $V$ is dense in $U$ and the restriction of $f$ to $V$ is locally constant.
\end{lemma}
\begin{proof}
	We first reduce to the case in which $f$ is continuous.
	There exists a definable dense open subset $U'$ of $U$ such that the restriction of $f$ to $U'$ is continuous by Proposition \ref{prop:Miller}(2).
	We may assume that $f$ is continuous by considering the restriction of $f$ to $U'$ instead of $f$ by using Proposition \ref{prop:Miller}(1).
	
	We set 
	\begin{align*}
		V &=\{x \in U\;|\; \text{there exists open box } B \subseteq U \text{ containing the point }x\\
		&\text{ such that } f \text{ is constant in }B\}.
	\end{align*}
	It is obvious that $V$ is open and definable.
	The remaining task is to show that $V$ is dense in $U$.
	
	Take an arbitrary point $x \in U$.
	We fix an arbitrary small box $B$ containing the point $x$ and contained in $U$.
	We have only to prove that $B \cap V \neq \emptyset$.
	It is obvious when $x \in V$.
	We assume that $x \notin V$.
	Set $Z=f(B)$.
	There exists an isolated point $z$ of $Z$ by Lemma \ref{lem:basic0} because $Z$ has an empty interior and the structure is d-minimal.
	Take a sufficiently small open interval $I$ containing the point $z$ such that $I \cap Z = \{z\}$.
	The inverse image $W=f^{-1}(I) \cap B$ of $f|_B$ is open because $f$ is continuous.
	The restriction of $f$ to $W$ is constant.
	It means that $W \subseteq V$ and this implies that $B \cap V \neq \emptyset$.
\end{proof}

We are now ready to prove Theorem \ref{thm:non_field_type}.
\begin{proof}[Proof of Theorem \ref{thm:non_field_type}]
	Recall that $(\mathcal R, D \cup E)^{\#}$ is d-minimal by Theorem \ref{thm:semi-d-minimal}.
	
	We first reduce to the case in which $m=1$.
	Assume that the theorem holds for $m=1$.
	Let $\pi_i:\mathbb R^m \to \mathbb R$ be the projection onto the $i$-th coordinate for $1 \leq i \leq m$.
	There exists a definable closed subset $X_i$ of $U$ with empty interior such that the restriction of $\pi_i \circ f$ to $U \setminus X_i$ is locally affine for every $1 \leq i \leq m$.
	The union $\bigcup_{i=1}^m X_i$ has an empty interior by Proposition \ref{prop:Miller}(1).
	The restriction of $f$ to $U \setminus \left(\bigcup_{i=1}^m X_i\right)$ is locally affine and $U \setminus \left(\bigcup_{i=1}^m X_i\right)$ is a dense open subset of $U$.
	We have succeeded in reducing to the case in which $m=1$.
	
	We next reduce to the case in which there exist a linear function $g$ and a subset $F$ of $\mathbb R$ definable in $(\mathcal R, D \cup E)^{\#}$ with empty interior such that, for every $x \in U$,  $f(x)=g(x)+d$ for some $d \in F$.
	By Lemma \ref{lem:locally_constant}, there exist finitely many pairs $(g_i:\mathbb R^n \to \mathbb R,F_i)_{i=1}^k$ of linear maps $g_i$ and subsets $F_i$ of $\mathbb R$ such that, 
	\begin{enumerate}
		\item[(a)] for each $x \in X$, the equality $f(x)=g_i(x)+d$ holds for some $1 \leq i \leq k$ and $d \in F_i$ and
		\item[(b)] the sets $F_i$ are definable in $(\mathcal R, D \cup E)^{\#}$ and have empty interiors for all $1 \leq i \leq k$.
	\end{enumerate}
	We define $W_i$ as follows for $1 \leq i \leq k$:
	\begin{align*}
		W_i&:=\{x \in U\;|\; \forall j \text{ with }\ 1 \leq j<i \ \forall d \in F_j\ g_j(x)+d \neq f(x) \\
		&\quad \text{ and } \exists d \in F_i\ g_i(x)+d=f(x)\}.
	\end{align*}
	We get the definable decomposition $U=\bigcup_{i=1}^kW_i$ of $U$ by property (a).
	Set $U_i=\myint(W_i)$ and $Z_i=W_i \setminus U_i$ for $1 \leq i \leq k$.
	Note that we may have $U_i=\emptyset$, but the following argument is valid even in the case which $U_i=\emptyset$.
	Each $Z_i$ has an empty interior and the union $\bigcup_{i=1}^k Z_i$ has an empty interior by Proposition \ref{prop:Miller}(1).
	If the theorem holds for every restriction of $f$ to $U_i$, there exists a definable  closed subset $Z'_i$ of $U_i$ such that the restriction of $f$ to $U_i \setminus Z'_i$ is locally affine and $Z'_i$ has an empty interior.
	The restriction of $f$ to $U \setminus \bigcup_{i=1}^m (Z_i \cup Z'_i)$ is locally affine and  the set $U \setminus \bigcup_{i=1}^m (Z_i \cup Z'_i)$ is a definable dense open subset of $U$ by Proposition \ref{prop:Miller}(1).
	We have reduced to the case in which there exist a linear function $g$ and a subset $F$ of $\mathbb R$ definable in $(\mathcal R, D \cup E)^{\#}$ with empty interior such that, for every $x \in U$,  $f(x)=g(x)+d$ for some $d \in F$.
	
	By Proposition \ref{prop:Miller}(3), we can find a definable function $\tau:U \to F$ such that $f(x)=g(x)+\tau(x)$ for every $x \in U$.
	We can take a definable dense open subset $V$ of $U$ such that the restriction of $\tau$ to $V$ is locally constant by Lemma \ref{lem:locally_constant}.
	The restriction of $f$ to $V$ is obviously locally affine.
\end{proof}

\section{Counterexamples}\label{sec:counterexamples}
Theorem \ref{thm:semi-d-minimal} does not hold if we drop the assumption that $E$ is compact.
We give two examples which illustrate it in this section.
The following proposition shows that Theorem \ref{thm:semi-d-minimal} does not hold if we drop the assumption that $E$ is bounded even when $E$ is closed. 
In fact, we have only to get $D$ as a bounded countably infinite subset of $\mathbb R$ which is not a union of finitely many discrete sets.
The Cantor ternary set is an example of such sets.

\begin{proposition}\label{prop:ex1}
	Let $D$ be a bounded countably infinite subset of $\mathbb R$.
	There exists a discrete closed subset $E$ of $\mathbb R$  defined by an increasing sequence of real numbers $(a_n)_{n \in \mathbb N}$ such that $D$ is definable in $(\mathbb R,<,+,0,E)$.
\end{proposition}
\begin{proof}
	Since $D$ is bounded, we may assume that $D$ is contained in $(0,\infty)$ by shifting $D$.
	We can take a natural number $N$ such that $D \subseteq (0,N)$ because $D$ is bounded.
	Since $D$ is countably infinite, we can take a sequence of real numbers $(d_n)_{n \in \mathbb N}$ so that $D=\{d_n\;|\; n \in \mathbb N\}$.
	We define a sequence of real numbers $(a_n)_{n \in \mathbb N}$ as follows:
	\[
	a_n = \left\{\begin{array}{lc}
	N(n+1) & \text{ if } n \text{ is odd,}\\
	Nn+d_{n/2} &  \text{ if } n \text{ is even.}
	\end{array}\right.
	\]
		
	We investigate the properties of the sequence $(a_n)_{n \in \mathbb N}$.
	It is obvious that $a_n$ is positive for every $n$ and the sequence is an increasing sequence.
	It is also obvious that $\lim_{n \to \infty}a_n=\infty$.
	Furthermore, the following claim holds:
	\medskip
	
	\textbf{Claim. } Let $n,m$ be natural numbers. The inequality $0<a_n-a_m <N$ holds if and only if $n=m+1$ and $m$ is odd.
	\begin{proof}[Proof of Claim]
		Since $(a_n)_{n \in \mathbb N}$ is an increasing sequence, we have two consequences.
		Firstly, it is obvious that $a_n-a_m>0$ if and only if $n>m$.
		Secondly, we have $a_n-a_m> a_{m+1}-a_m$ when $n>m+1$.
		It is easy to check that $a_{m+2}-a_m>N$ using the inequalities $-N<d_{k+1}-d_k<N$ for every $k$.
		Therefore, we have only to check that $a_{m+1}-a_m<N$ if and only if $m$ is odd.
		
		We first consider the case in which $m$ is odd.
		We have $a_{m+1}-a_m =N(m+1)+d_{(m+1)/2}-N(m+1)=d_{(m+1)/2}<N$.
		When $m$ is even, we have $a_{m+1}-a_m =N(m+2)-(Nm+d_{m/2})=2N-d_{m/2}>N$ because $d_{m/2}<N$.
		We have proven the claim.
	\end{proof}
	
	We have finished the preparation.
	We set $E:=\{a_n\;|\;n \in \mathbb N\}$, which is the set we are looking for, and we set $\mathcal R=(\mathbb R,<,+,E)$.
	Since $(a_n)_{n \in \mathbb N}$ is an increasing sequence and $E:=\{a_n\;|\;n \in \mathbb N\}$, it is easy to show that $E$ is discrete and closed.
	We omit the details.
	Consider the formula $$\phi(x)= (0<x<N) \wedge (\exists x_1, x_2 \in E\ x=x_1-x_2).$$
	We prove that $D$ is defined by the formula $\phi(x)$ in $\mathcal R$.
	When $x \in D$, there exists a natural number $n$ such that $d_n=x$.
	Set $x_1:=a_{2n}=2Nn+d_n$ and $x_2:=a_{2n-1}=2Nn$.
	We have $x_1,x_2 \in E$ and $x=x_1-x_2$.
	It is obvious that $0<x<N$.
	It implies that $\mathcal R \models \phi(x)$.
	
	We next consider the case in which $\mathcal R \models \phi(x)$.
	By the definition of the formula $\phi(x)$, we have $0<x<N$ and there exist $x_1, x_2 \in E$ such that $x=x_1-x_2$.
	There are two natural numbers $n$ and $m$ with $x_1=a_n$ and $x_2=a_m$ by the definition of $E$.
	By Claim, the natural number $m$ is odd and $n=m+1$ because $0<a_n-a_m<N$.
	It implies that $x=a_{m+1}-a_m =N(m+1)+d_{(m+1)/2}-N(m+1)=d_{(m+1)/2} \in D$.
	We have proven that $D$ is defined by the formula $\phi(x)$.
\end{proof}

Theorem \ref{thm:semi-d-minimal} does not hold if we drop the assumption that $E$ is closed even when $E$ is bounded and discrete.
In fact, let $E$ be a bounded discrete set whose closure is not a union of finitely many discrete sets.
This implies that $(\mathbb R,<,+,0,E)$ is not d-minimal because the closure is definable in the structure.
The following is an example of such a set $E$.
\begin{example}
	Let $\mathfrak S$ be the set of sequences of elements in $\{0,1,2\}$ of finite length defined so that $\sigma=(i_1,\ldots, i_n) \in \mathfrak S$ if and only if
	\begin{itemize}
		\item $i_n \neq 0$;
		\item $i_j \neq 1$ for each $1 \leq j  < n$.
	\end{itemize} 
	We have $i_n = 1,2$ when $\sigma=(i_1,\ldots, i_n) \in \mathfrak S$.
	
	We fix $\sigma=(i_1,\ldots, i_n) \in \mathfrak S$.
	We put $u(\sigma)=\sum_{j=1}^{n-1} \frac{i_j}{3^j}$ and $v(\sigma)=\sum_{j=1}^n \frac{i_j}{3^j}.$
	For any $k \in \mathbb N$, we set $a_{\sigma,k}=v(\sigma)- \frac{(-1)^{i_n} }{k \cdot 3^{n+1}}$ and $E_{\sigma}=\{a_{\sigma,k}\;|\; k \in \mathbb N\}$.
	We also put $I_{\sigma}=(u(\sigma)+\frac{1}{3^n}, u(\sigma)+\frac{1}{3^n}+\frac{1}{3^{n+1}})$ when $i_n=1$ and $I_{\sigma}=(u(\sigma)+\frac{1}{3^n}+\frac{2}{3^{n+1}}, u(\sigma)+\frac{2}{3^n})$ when $i_n=2$.
	It is easy to prove that 
	\begin{equation}
	I_{\sigma} \cap I_{\sigma'}= \emptyset \label{eq:empty}
	\end{equation}
	 if $\sigma, \sigma' \in \mathfrak S$ and $\sigma \neq \sigma'$.
	It is also obvious that $E_{\sigma}$ is contained in the open interval $I_{\sigma}$.
	
	Put $D=\{v(\sigma)\;|\; \sigma \in \mathfrak S\}$ and $E=\bigcup_{\sigma \in \mathfrak S}E_{\sigma}$.
	Then, 
	\begin{itemize}
		\item the set $D$ does not have an isolated point;
		\item the set $E$ is discrete and bounded;
		\item the closure of $E$ contains the set $D$.
	\end{itemize}
	The first and third assertions imply that the closure of $E$ is not a union of finitely many discrete sets by Lemma \ref{lem:basic0} and Lemma \ref{lem:basic2}.
	
	We prove them.
	We first show easy  assertions.
	The set $E$ is obviously bounded.
	Indeed, it is contained in the open interval $(0,1)$.
	We next demonstrate that $E$ is discrete. 
	Take an arbitrary element $x \in E$.
	The point $x$ belongs to $E_{\sigma}$ for some $\sigma \in \mathfrak S$.
	We have $E \cap I_{\sigma}=E_{\sigma}$ by the definition of $E$ and equality (\ref{eq:empty}).
	Since $E_{\sigma}$ is discrete, we can take an open interval $I$ containing the point $x$ such that $E_ {\sigma} \cap I = \{x\}$.
	We have $E \cap I \cap I_{\sigma} = \{x\}$.
	It means that $x$ is an isolated point in $E$ and $E$ is a discrete set.
	Since $\lim_{k \to \infty}a_{\sigma,k}=v(\sigma)$, the set $D$ is contained in the closure of $E$.
	
	The remaining task is to show that $D$ does not have an isolated point.
	Let $x$ be an arbitrary point in $D$.
	We show that $x$ is not isolated.
	The point $x$ is of the form $x=v(\sigma)$ for some $\sigma=(i_1,\ldots, i_n) \in \mathfrak S$.
	We first consider the case in which $i_n = 1$.
	For any $k$, let $\tau_k$ be the concatenation of $(i_1,\ldots, i_{n-1})$ with the sequence of length $k$ whose first element is zero and others are twos.
	It is easy to check that $\tau_k \in \mathfrak S$.
	We omit it.
	The sequence $(v(\tau_k))_{k \in \mathbb N}$ is a Cauchy sequence converging to $v(\sigma)$.
	It implies that $x$ is not an isolated point in this case.
	We next consider the case in which $i_n = 2$.
	For any $k$, let $\tau'_k$ be the concatenation of $\sigma$ with the sequence of length $k$ whose last element is one and others are zeros.
	We have $\tau'_k \in \mathfrak S$ and $(v(\tau'_k))_{k \in \mathbb N}$ is a Cauchy sequence converging to $v(\sigma)$.
	The point $x$ is not isolated in $D$ also in this case.
\end{example}

\section{Concluding remarks}\label{sec:concluding}
We have proven that the expansion $(\mathbb R, <,+,0,\Lambda,E)$ of the $\mathbb R$-vector space structure over $\mathbb R$ by a Cauchy sequence $E$ is d-minimal.
Here, $\Lambda$ denotes the set of all real scalar functions on $\mathbb R$.
Our example is an extremely tame structure. 
Every definable continuous map on an open set is locally affine in a dense open subset of the domain.
The structures $(\mathbb R, <,+,0,\Lambda,E)$ falls into `type A, affine' category under the classification system in \cite{HW}. 

The expansion of the ordered field of reals by the Cauchy sequence $\{1/n\;|\; n \in \mathbb N\}$ is interdefinable with $(\mathbb R,<,+,\cdot,0,1,\mathbb Z)$.
Every continuous function is definable in the structure $(\mathbb R,<,+,\cdot,0,1,\mathbb Z)$ (See \cite[Chapter 1(2.6)]{vdD}) and it falls into `type C' under the classification system in \cite{HW}. 
The expansion of the ordered field of reals by a Cauchy sequence is possibly wild.

A natural question is whether there exists a d-minimal expansion of the ordered group of reals in which a given Cauchy sequence is definable and which has an intermediate tameness fallen into `type A, field-type' category under the classification system in \cite{HW}. 
The author has neither an affirmative nor negative witness to this question.

\end{document}